\documentclass[a4paper]{article}
\usepackage{amsthm,amssymb,amsmath,enumerate,graphicx,epsf}
\usepackage{bbold}
\usepackage{mathtools,enumitem}
\usepackage{authblk}
\usepackage[percent]{overpic}
\usepackage{graphicx}
\RequirePackage{latexsym} \RequirePackage{amsthm}
\RequirePackage{amsmath}
\usepackage{authblk}
\usepackage{tikz}
\usepackage[dvipsnames]{xcolor}
\usepackage{subcaption}
\usepackage{geometry}
\usepackage{enumerate}
\RequirePackage{amssymb} \RequirePackage{makeidx}
\usepackage{dsfont}
\usepackage{mathrsfs}

\newcommand{\COLORON}{1}
\newcommand{\NOTESON}{0}
\newcommand{\Debug}{0} %CHANGE TO 0 TO REMOVE ALL COMMENTS ETC.

\usepackage{amsthm,amssymb,amsmath,bbm,enumerate,graphicx,epsf,stmaryrd,accents}

\usepackage{authblk}

\newcommand{\comment}[1]{}
\newcommand{\COMMENT}[1]{}

\colorlet{darkishRed}{red!60!black}
\definecolor{darkgray}{rgb}{0.3,0.3,0.3}
\definecolor{lgreen}{rgb}{0.065, 0.225, 0.065}
\newcommand{\defi}[1]{{\color{darkishRed}{\emph{#1}}}}
\newcommand{\defim}[1]{{\color{darkishRed}{#1}}}

\newcommand{\FIG}{0}

\ifnum \NOTESON = 1 \newcommand{\note}[1]{ 

\hspace*{-30pt}
	{\color{blue}  NOTE: \color{Turquoise}{\small  \tt \begin{minipage}[c]{1.1\textwidth}  #1 \end{minipage} \ignorespacesafterend }} 
	
	}
\else \newcommand{\note}[1]{} \fi

\newcommand{\afsubm}[1]{ \ifnum \Debug = 1 {\mymargin{#1}}
\fi} %For notes on after-submission changes

\ifnum \Debug = 1 
\else  \fi

\ifnum \FIG = 1 
\else  \fi

\ifnum \FIG = 1 
\else  \fi

\ifnum \Debug = 1 \usepackage[notref,notcite]{showkeys}
\fi

\ifnum \COLORON = 0 \renewcommand{\color}[1]{}
\fi

\newcommand{\N}{\ensuremath{\mathbb N}}
\newcommand{\R}{\ensuremath{\mathbb R}}

\newcommand{\cp}{\ensuremath{\mathcal P}}

\newcommand{\cv}{\ensuremath{\mathcal V}}
\newcommand{\cx}{\ensuremath{\mathcal X}}

\makeatletter
\DeclareRobustCommand{\cev}[1]{%
  \mathpalette\do@cev{#1}%
}
\newcommand{\do@cev}[2]{%
  \fix@cev{#1}{+}%
  \reflectbox{$\m@th#1\vec{\reflectbox{$\fix@cev{#1}{-}\m@th#1#2\fix@cev{#1}{+}$}}$}%
  \fix@cev{#1}{-}%
}
\newcommand{\fix@cev}[2]{%
  \ifx#1\displaystyle
    \mkern#23mu
  \else
    \ifx#1\textstyle
      \mkern#23mu
    \else
      \ifx#1\scriptstyle
        \mkern#22mu
      \else
        \mkern#22mu
      \fi
    \fi
  \fi
}

\makeatother
\newcommand{\nin}{\ensuremath{{n\in\N}}}

\newcommand{\pth}[2]{\ensuremath{#1}\text{--}\ensuremath{#2}~path}

\newcommand{\g}{\ensuremath{G\ }}
\newcommand{\G}{\ensuremath{G}}

\newcommand{\fe}{for every}

\newcommand{\st}{such that}

\newcommand{\wrt}{with respect to}

\newcommand{\labtequ}[2]{%\labtequc{#1}{#2}}
 \begin{equation} \label{#1} 	\begin{minipage}[c]{0.9\textwidth}  #2 \end{minipage} \ignorespacesafterend \end{equation} }

\newcommand{\mymargin}[1]{% <- dieses % verhindert ein ungewolltes Leerzeichen
 \ifnum \Debug = 1
  \marginpar{%
    \begin{minipage}{\marginparwidth}\small%
      \begin{flushleft}%
        {\color{blue}#1}%
      \end{flushleft}%
   \end{minipage}%
  }%
 \fi
}%

\newcommand{\extras}[1]{% <- dieses % verhindert ein ungewolltes Leerzeichen
 \ifnum \Debug = 1
\section{Extras} #1
 \fi
}%

\newcommand{\mySection}[2]{}

\usepackage{thmtools, mathtools}
\colorlet{darkishGreen}{green!50!black}
\colorlet{darkishBlue}{blue!60!black}
\usepackage[colorlinks=true, citecolor=darkishGreen, linkcolor=darkishBlue, urlcolor=darkishBlue, bookmarks,hyperfootnotes=false, psdextra=true,hypertexnames=false]{hyperref}
\usepackage[nameinlink, capitalise, noabbrev]{cleveref}
\usepackage{nameref}

\newtheorem{proposition}{Proposition}[section]

\newtheorem{theorem}[proposition]{Theorem}

\newtheorem{lemma}[proposition]{Lemma}

\newtheorem{problem}[proposition]{{Problem}}

\newtheorem{examp}[proposition]{Example}%[section]

\crefname{lemma}{Lemma}{Lemmas}

\theoremstyle{definition}
\newtheorem{construction}[proposition]{Construction}

\newcommand{\diam}{\mathrm{diam}}

\newcommand{\TD}{tree-decomposition}
\newcommand{\Id}{\ensuremath{\mathrm{Id}}}

\newcommand{\ncm}{near-component}

\def\td{tree-decom\-posi\-tion}

\newsavebox{\otterbox}
\sbox{\otterbox}{\includegraphics[height=1em]{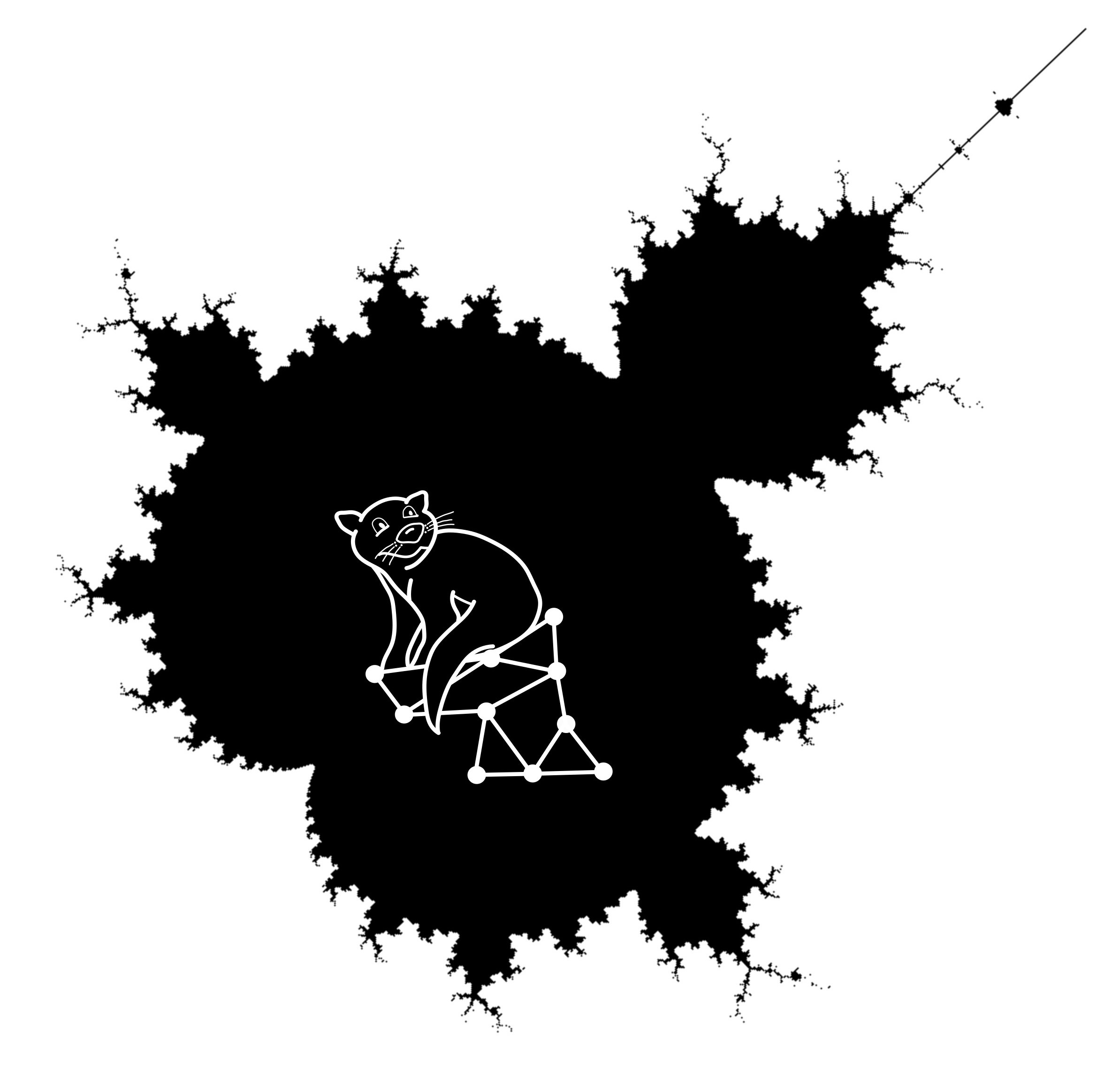}}

\makeatletter

\renewcommand{\@fnsymbol}[1]{%
  \ifcase#1\or
    \raisebox{-0.3ex}{\usebox{\otterbox}}%
  \or
    \dagger
  \or
    \ddagger
  \or
    \S
  \or
    \P
  \else
    *
  \fi
}

\makeatother

\begin{document}

\title{A coarse block-cutvertex tree-decomposition}

\author[1]{Sandra Albrechtsen\thanks{Supported by the Alexander von Humboldt Foundation in the framework of the Alexander von Humboldt Professorship of Daniel Král' endowed by the Federal Ministry of Education and Research.}}

\affil[1]{{Institute of Mathematics}\\ {Leipzig University}\\ {Augustusplatz 10}\\ {04109 Leipzig}\\ {Germany}}

\author[2]{Agelos Georgakopoulos}

\affil[2]{{Mathematics Institute}\\ {University of Warwick}\\  {CV4 7AL, UK}}

\date{}
\maketitle

\begin{abstract}
    We obtain a coarse version of the block-cutvertex tree-decomposition of a\\ connected graph.
\end{abstract}


\maketitle

\section{Introduction}

This paper provides a coarse version of the fundamental fact that every connected graph admits a block-cutvertex tree-decomposition. We thereby follow the recent trend of coarse graph theory, which seeks to translate the basic notions and theorems of graph theory into a metric-space language. Motivation comes from  making graph-theoretic results  robust under local perturbations —specifically, invariant under quasi-isometries— and also making them applicable to metric spaces. See \cite{ADGSmallCounterexamples,AJKW,BLPPSep,DGHLMGallai,FujPapCoa,LiuWeakCoarseMenger,McMFat,NgScSeCoa,NgScSeAsyIV} for a selection of results, and \cite{GeoPapMin} for a discussion of the motivation.  
\medskip

The natural coarse analogue of a cutvertex is a separator of bounded diameter. But the right definition of a `coarse block' is less clear: we cannot just take it to be a subgraph $H$ of a graph \g that cannot be separated by a set of bounded diameter, since the neighbourhood $N(x)$ of any vertex $x$ of $H$ has diameter at most 2 and separates $x$ from the rest of $H$. But there is a simple way to overcome this difficulty, namely to only consider separators ---of diameter at most some $d\in \N$--- that separate two or more components of diameter at least some parameter $D$ depending on $d$. More precisely, we say that 
a vertex set $U \subseteq V(G)$ is \defi{$(d,D)$-inseparable in $G$} if for every set $S \subseteq V(G)$ of diameter at most~$d$ in~$G$ and every two vertices $u,v \in U$, either one of $u,v$ has distance at most~$D$ from $S$ in $G$, or $u,v$ lie in the same component of $G-S$. We can now state our main result:

\begin{theorem} \label{main:CoarseBlockCutTree1}
    Let $d \in \N$. Every graph $G$ admits a tree-decomposition \st: 
    \begin{enumerate}[label=\rm{(\roman*)}]
        \item \label{itm:CoarseBlockCutTree1:i} every adhesion set has diameter at most $5d+2$ in $G$, and
        \item \label{itm:CoarseBlockButTree1:ii} every bag is $(d,2d+1)$-inseparable in $G$.
    \end{enumerate}
\end{theorem}

The classical block-cutvertex tree of a graph $G$ is straightforward to construct, as the decomposition is unique, being determined by the blocks of \G. In contrast, \cref{main:CoarseBlockCutTree1} is more difficult. Apart from the fact that one needs to come up with the right notion of `coarse block' as discussed above, the separators and coarse blocks can overlap in complex ways, making the construction non-canonical, and the interplay between the parameters $d$ and $D$ needs to be controlled. 

We remark that by the coarse Menger Theorem for two paths~\cite{AHJKW,GeoPapMin}, property~\ref{itm:CoarseBlockButTree1:ii} implies that for every two sets $A,B$ of vertices of $G$ that are contained in a common bag and that have diameter at least $5d+3$ in $G$, there are two $A$--$B$ paths in $G$ that are at distance at least~$d/129$ from each other. Therefore, one can think of the bags as being `coarsely $2$-path-connected'.

A  result similar to \cref{main:CoarseBlockCutTree1} has been obtained independently by Balig\'acs, Bla{\v{z}}ej, Czy{\.{z}}ewska, Pilipczuk and Protopapas \cite{BBCPPCoarseBlockCutTree}. We remark that their constants are slightly better than ours.

\medskip
A perhaps unsatisfactory feature of \cref{main:CoarseBlockCutTree1} is that inseperability of the bags in \ref{itm:CoarseBlockButTree1:ii} is defined \wrt\ the metric of the ambient graph \G, rather than being a property of the bag itself. However, by modifying our tree-decomposition slighlty, namely by enlarging its bags up to a bounded-distance neighbourhood, we will obtain the following variant, whereby the intrinsic metric of each bag is very similar to that it inherits from \G, and each bag ---or rather torso--- is `coarsely inseperable' as a graph. To make this more precise, we define the \defi{outer-torso} \defi{$O_t$} of a bag $V_t$ to be the graph obtained from $G$ by contracting each component of $G-V_t$ into a vertex. We will prove

\begin{theorem} \label{thm:cbct}
    Let $d \in \N$. Every graph $G$ admits a tree-decomposition $(T, \cv)$ \st: 
    \begin{enumerate}[label=\rm{(\roman*)}]
        \item \label{itm:TD i} every adhesion set has intrinsic diameter at most~$15d+8$, 
        \item \label{itm:TD iii} every outer-torso is a $(d,7d+4)$-inseparable graph\footnote{A graph $G=(V,E)$ is \defi{$(d,D)$-inseparable}, if $V$ is $(d,D)$-inseparable according to the above definition.}, and 
        \item \label{itm:TD ii} for every bag $V_t$, the identity $\Id: G[V_t] \to G$ has additive distortion at most $20d+12$.
    \end{enumerate}
\end{theorem}

Roughly speaking, the block-cutvertex tree decomposes a connected graph along its cutvertices into its $2$-connected pieces.
Following this idea, the Tutte \td~\cite{Tutte} decomposes a $2$-connected graph $G$ along adhesion sets of size~$2$ into torsos that are $3$-connected, a cycle, or a~$K_2$. An interesting question would be whether there is a coarse version of this decomposition.

\begin{problem}
    Find a coarse version of the Tutte decomposition.    
\end{problem}

We intentionally leave the problem somewhat vague. Ideally, we would like to start with a `coarsely $2$-connected' graph~$G$ (i.e.\ $(d,D)$-inseparable for some (large) parameters $d,D \in \N$) and then find a \td\ of~$G$ whose adhesion sets can be covered by two sets of bounded diameter, and whose bags are `coarsely $3$-connected' in the following sense: For any two vertices $u,v$ that lie in a common bag, and for every set $S \subseteq V(G)$ that can be covered by two sets of diameter at most~$d'$, either one of $u,v$ is at distance at most~$D'$ from~$U$ or $u,v$ lie in the same component of~$G-S$.

\section{Preliminaries}

A vertex $v$ of a graph $G$ is a \defi{cutvertex}, if $G-v$ has more components than $G$. 
A \defi{block} of $G$ is a maximal connected subgraph that does not have a cutvertex.

The \defi{block-cutvertex tree} of $G$ is the tree $T$ such that $V(T)$ is the set of blocks and cutvertices of $G$, and a block sends an edge in $E(T)$ to each cutvertex it contains. 

\subsection{Distances, diameter, and balls}

Let $G$ be a graph.
We write~\defi{$d_G(v, u)$} for the distance of the two vertices~$v$ and~$u$ in~$G$. 
For two sets~$U$ and~$U'$ of vertices of~$G$, we write~\defi{$d_G(U, U')$} for the minimum distance of two elements of~$U$ and~$U'$, respectively.

Let $U$ be a set of vertices of~$G$.
The \defi{diameter}~\defi{$\diam_G(U)$} of $U$ in~$G$ is the smallest number~$d \in \N \cup \{\infty\}$ such that $d_G(u,v) \leq d$ for all $u,v \in U$. The \defi{intrinsic} diameter of~$U$ is the smallest $d \in \N \cup \{\infty\}$ such that $d_{G[U]}(u,v) \leq d$ for all $u,v \in U$.

The \defi{ball (in~$G$) around~$U$ of radius $r \in \N$}, denoted by~\defi{$B_G(U, r)$}, is the set of all vertices in~$G$ at distance at most~$r$ from~$U$ in~$G$.

Let $G,H$ be graphs. A map $\varphi: V(G) \to V(H)$ has \defi{additive distortion} at most~$D \in \N$ if $d_G(u,v) - D \leq d_H(\varphi(u),\varphi(v)) \leq d_G(u,v)+D$.

\subsection{Tree-decompositions}
\label{sec TD}

A \defi{tree-decomposition} of a graph~$G$ is a pair $\mathcal{T} = (T, (V_t : t \in V(T)))$, where $T$ is a tree, and $V_t$ is a subset of $V(G)$ for each $t \in V(T)$ such that:
\begin{enumerate}[label=\rm{(T\arabic*)}]
    \item \label{TDi} $V(G) = \bigcup_{t \in V(T)} V_t$;
    \item \label{TDii}  for every edge $e = uv$ of $G$, there exists $t \in V(T)$ with $u, v \in V_t$; and
    \item \label{TDiii} for all $t_1, t_2, t_3 \in V(T)$, if $t_2$ lies on the path of $T$ between $t_1, t_3$, then $V_{t_1} \cap V_{t_3} \subseteq V_{t_2}$.
\end{enumerate}
The sets $V_t$ are the \defi{bags} of $\mathcal{T}$, and $T$ is its \defi{decomposition tree}. The sets $\defim{V_e} := V_s \cap V_t$ for edges $e=st$ of~$T$ are the \defi{adhesion sets} of~$\mathcal{T}$. 
\medskip

Given $X\subset V(G)$ and $r\in \R_+$, let $\defim{X^{+r}} := B_G(X,r)$.
\begin{lemma} \label{lem:TD exp}
    Let $(T, (V_t : t \in V(T)))$ be a tree-decomposition of a graph~$G$, and $R\in \R_+$. Then $(T, (V_t^{+R} : t \in V(T)))$ is also a tree-decomposition of~$G$. 

    If for each adhesion set $V_e$ we have $\diam_G(V_e) < r \in \N$, then for every $R \geq \lfloor r/2\rfloor$ and $t \in V(T)$ the identity $\Id: G\left[V_t^{R}\right] \to G$ has additive distortion at most $4R$.
    
    Moreover, there exists a surjective map from the outer-torso $O^{+R}_t$ of $V^{+R}_t$ to $G\left[V^{+R}_t\right]$ of additive distortion at most~$4R$.\footnote{A very similar lemma appears in \cite{GeoMolQua}.}
\end{lemma}

\begin{proof}
    The first part of the statement is \cite[Lemma~3.3]{ADEFJKW}.

    Let $u,v \in V^{+R}_t$. Set $\defim{G'} := G\left[V^{+R}_t\right]$. 
    Since $G' \subseteq G$, we have $d_G(u,v) \leq d_{G'}(u,v)$.
    
    If $u,v \in V_t$, then, since each adhesion set has diameter less than~$r$, every shortest $u$--$v$~path in~$G$ is contained in $G[B_G(V_t, \lfloor r/2\rfloor)] \subseteq G'$, and thus $d_{G'}(u,v) \leq d_G(u,v)$. 
    
    Otherwise, let $u',v' \in V_t$ such that $d_{G'}(u,u'), d_{G'}(v,v') \leq R$. Then 
    \begin{align*}
    d_{G'}(u,v) &\leq d_{G'}(u,u') + d_{G'}(u',v') + d_{G'}(v',v) \leq R + d_G(u',v') + R = d_G(u',v') + 2R \\
    &\leq (d_G(u',u) + d_G(u,v) + d_G(v,v')) + 2R \leq (R + d_G(u,v) + R) + 2R,
    \end{align*}
    and thus $d_{G'}(u,v) \leq d_G(u,v) + 4R$ as desired.
    \smallskip

    For the `moreover'-part, set $\defim{O'} := O^{+R}_t$ and define $\varphi: O' \to G'$ by $\varphi(v) = v$ for all $v \in V(G')$ and $\varphi(v) = v'$ for all $v \in V(O' - G')$ where $v' \in N_{O'}(v)$. In particular, $v' \in V(G')$. Since $G' \subseteq O'$, we have $d_{O'}(u,v) \leq d_{G'}(\varphi(u), \varphi(v))+2$. 

    By the same argument as above, we also have $d_{G'}(\varphi(u), \varphi(v)) \leq d_{O'}(u,v)$ for all $u,v \in V_t$. Otherwise, similar as above, $d_{G'}(\varphi(u), \varphi(v)) \leq d_{O'}(u,v) + 4R$.
\end{proof}

\section{The construction}

In this section we prove \cref{main:CoarseBlockCutTree1,thm:cbct}. We first prove \cref{main:CoarseBlockCutTree1}. We then take the \td\ that we constructed for \cref{main:CoarseBlockCutTree1}, expand its bags slightly, and then show that the arising \td\ is as desired for \cref{thm:cbct}.
\medskip

Let $G,H$ be graphs. A \defi{graph-partition of}~$G$ \defi{over}~$H$, or $H$-\defi{partition} for short, is a partition $\cx := (X_h \mid h \in V(H))$ of~$V(G)$ (into \defi{boxes}~$X_h$) indexed by the vertices of~$H$ such that for every edge $uv \in E(G)$, if $u \in X_{g}$ and $v \in X_h$, then $g = h$ or $gh \in E(H)$.
\medskip

Our construction of a coarse block-cutvertex \td\ of~$G$ for the proof of \cref{main:CoarseBlockCutTree1} is based on finding an appropriate $H$-{partition} of~$G$ so that $H$ can be thought of as a coarse-grained variant of~$G$, respecting its coarse-cutvertex structure: every box corresponding to a cutvertex of~$H$ will separate~$G$, and every block of~$H$ will induce a $(d,2d+1)$-inseparable subgraph of~$G$ (see \cref{lem:blocks}). However, there may be distinct blocks $B,B'$ of~$H$ whose union still induces a $(d,2d+1)$-inseparable subgraph of~$G$; this happens when the box of some cutvertex~$h$ of $H$ intersects the boxes of both~$B$ and~$B'$ in a subset of large diameter compared to~$d$. Thus, the tree~$T$ of \cref{main:CoarseBlockCutTree1} will be obtained from the block-cutvertex tree~$T'$ of~$H$ after contracting some edges of~$T'$ (incident with such cutvertices~$h$).
\medskip

We now start with the proof of \cref{main:CoarseBlockCutTree1}. For this, we may assume that $G$ is connected; for otherwise, we can obtain the desired \td\ of~$G$ by decomposing each of its components individually, and then gluing the \td\ together.

We begin by formally defining the aforementioned $H$-{partition}:

\begin{construction} \label{constr:HPartition}
    Pick some $\defim{o} \in V(G)$. For $\nin$, let $A_n$ be the $n$-th $d$-wide annulus of \G\ \wrt\ $o$, i.e.\ 
    \[
    \defim{A_n} := \{ v\in V(G) \mid d(o,v)\in [nd, (n+1)d)\}.
    \]

    A set $U \subseteq V(G)$ is \defi{$r$-near-connected}, for $r\in \N$, if for every $x,y\in U$, there is a sequence $x=x_0,x_1, \ldots, x_k=y$ of vertices in $U$ such that $d_G(x_i,x_{i+1})\leq r$ \fe\ $i<k$.
    An \defi{$r$-\ncm} of $U$ is a maximal subset of $U$ that is $r$-near-connected.

    For each $n \in \N$ and each component $C$ of $G- \bigcup_{i<n} A_i$, we declare each $2(d+1)$-\ncm\ of $A_n\cap C$ to be a \defi{box} of the partition. We obtain \defi{$H$} from \g by contracting each box into a vertex. Thus each vertex $h$ of $H$ corresponds to a unique box \defi{$X_h$}. Let $\defim{\cx}:= (X_h \mid h \in V(H))$.
    Since $G$ is connected, $\cx$ is a graph-partition of~$G$ over~$H$.
\end{construction}

(Such `layered' partitions are a common tool in coarse graph theory; see e.g.\ \cite{ADGK2t,CDNRV,FujPapAsy,GeoPapMin}. But considering each component of $G- \bigcup_{i<n} A_i$ separately is a subtle point of our construction, without which our proof fails. A similar idea has been used in \cite{ADGK2t}.)
\medskip

Given $v \in V(G)$, let \defi{$X(v)$} be the unique box $X_h \in \cx$ that contains $v$, and let \defi{$h(v)$} be the corresponding node of $H$, i.e.\ $X(v) = X_{h(v)}$.

\begin{lemma} \label{lem:blocks}
    Let $G$ be a graph, and let $H, \cx$ be given by \Cref{constr:HPartition}. For every two $X_g, X_h \in \cx$ such that either $g=h$ or $g,h$ lie in the same block of $H$, every two vertices $u \in X_g$ and $v \in X_h$ are $(d,2d+1)$-inseparable in $G$.
\end{lemma}
%%%%%%%%%%%%%%%
\begin{proof}
    We prove that for every such two $u,v$, and for every set $S \subseteq V(G)$ of diameter at most~$d$, if $u,v$ both have distance more than~$2d+1$ from~$S$, then there is a path in~$G$ avoiding $S$ that connects $u$ to either $v$ or $o$. By applying this to both $u$ and $v$, we either obtain a $u$--$v$ path in $G$ avoiding $S$, or we obtain two paths in $G$ avoiding~$S$, one between $u$ and $o$ and one between $v$ and $o$, whose concatenation then yields a $u$--$v$ walk in~$G$ avoiding~$S$. In both cases, $u,v$ are not separated by~$S$. Since $S$ was arbitrary, this shows that $u,v$ are $(d,2d+1)$-inseparable.
    \smallskip

    So let $u,v,S$ be given and assume that $d_G(\{u,v\}, S) \geq 2d+2$.
    Let $\defim{i} \in \N$ be maximal such that $S$ meets the $i$-th annulus~$A_i$. Note that since each annulus has height~$d$, and $S$ has diameter at most~$d$, the set $S$ can meet at most two layers, and these then need to be consecutive. Thus, $S \subseteq A_i \cup A_{i-1}$. Let $\defim{j} \in \N$ be such that $u \in A_j$. We distinguish two cases: 
    \medskip
    
    \noindent \textbf{Case 1}: $j \leq i$.

    Let $P$ be a shortest path in $G$ from $u$ to $o$, and let $p$ be the last vertex of $P$ when moving along $P$ from $u$ to $o$ that is contained in $A_j \cup A_{j-1}$. Further, let $q$ be the next vertex of $P$. Since $P$ is a shortest path from~$u$ to~$o$ and $i \geq j$ by the assumption of Case~1, the subpath~$qPo$ of~$P$ between $q$ and $o$ avoids $A_i \cup A_{i-1}$, and hence it avoids~$S$.
    Moreover, again because $P$ is shortest, and every annulus has height~$d$, the subpath $uPp$ of $P$ between $u$ and $p$ has length at most $2d+1$. 
    Since $u$ has distance at least $2d+2$ from $S$, it follows that $uPp$ avoids $S$, too. Hence $P$ avoids $S$ as desired.
    \medskip

    \noindent \textbf{Case 2}: $j > i$.

    Let \defi{$C^i_u$} be the component of $G- \bigcup_{n \leq i} A_n$ containing $u$, and let \defi{$C^{i-1}_u$} be the component of $G-\bigcup_{n \leq i-1} A_n$ containing $C^i_u \ni u$. (Note that $C^i_u, C^{i-1}_u$ exist because $j > i$ by the assumption of Case~2.) We distinguish two cases:
    \medskip

    \begin{figure}[ht]
        \centering
        \begin{subfigure}[b]{0.45\linewidth}
            \centering
            \includegraphics[width=1\linewidth]{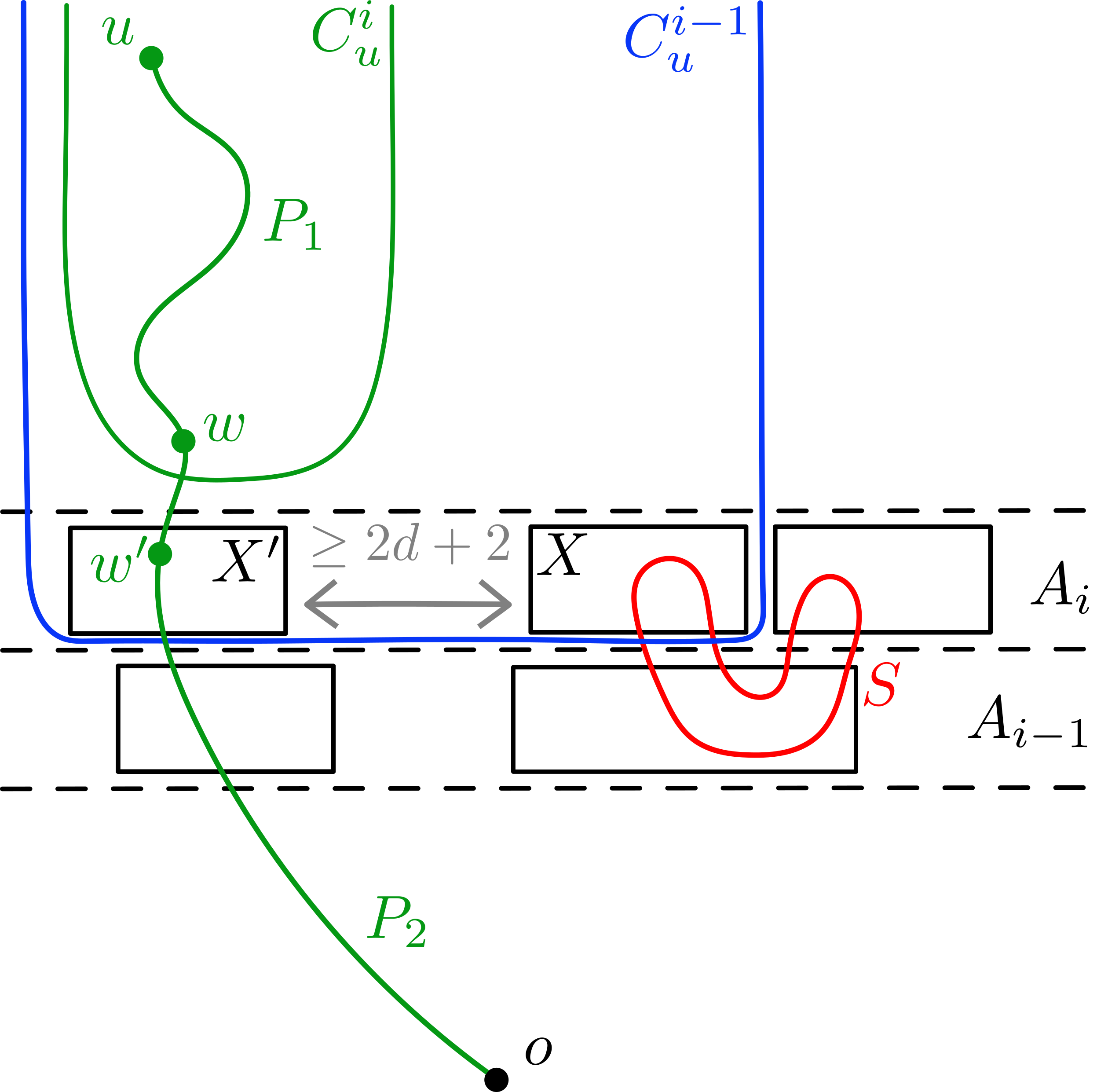}
            \caption{Case 2a.}
            \label{fig:BlocksAreInseparable:Case2a}
        \end{subfigure}
        \hspace{2em}
        \begin{subfigure}[b]{0.45\linewidth}
            \centering
            \includegraphics[width=1\linewidth]{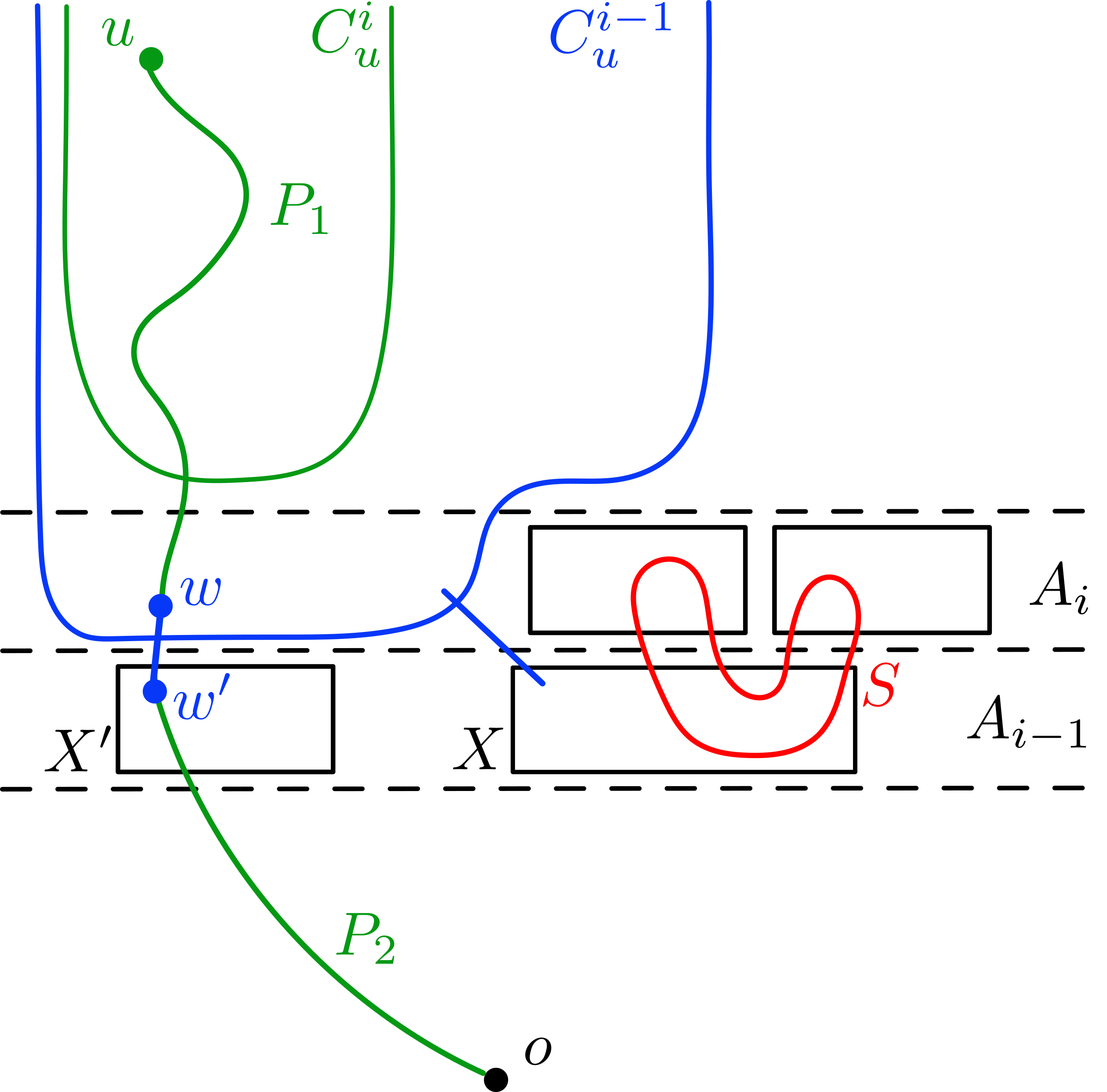}
            \caption{Case 2b}
            \label{fig:BlocksAreInseparable:Case2b}
        \end{subfigure}
        \caption{An illustration of the situation in the proof of \cref{lem:blocks}.}
    \end{figure}

    \noindent \textbf{Case 2a}: \emph{$C^{i-1}_u$ meets $S$.} 

    The reader may look at \cref{fig:BlocksAreInseparable:Case2a} to follow the proof in Case~2a more easily.
    
    By \Cref{constr:HPartition}, $C^{i-1}_u$ is a union of boxes $X \in \cx$. As $C_u^{i-1}$ meets $S$ by the assumption of Case~2a, there is some $\defim{X} \in \cx$ such that $X \subseteq V(C_u^{i-1})$ and $S \cap X \neq \emptyset$. In particular, $X \subseteq A_i$. Recall that $S$ is contained in $A_i \cup A_{i-1}$, and thus $C^{i}_u$ avoids~$S$ and~$X$. 
    If $v$ lies in~$C^i_u$, then there exists a $u$--$v$ path in~$C^i_u$, which hence avoids $S$. 
    Thus, we may assume that $v$ lies outside of $C^{i}_u$. Since $u,v$ lie inside boxes belonging to the same block of~$H$, i.e.\ $h(u), h(v)$ lie in the same block of~$H$, they cannot be separated in~$G$ by removing a single box of~$\cx$. In particular, $C^i_u$ sends edges to at least one box \defi{$X'$} in~$A_i$ other than~$X$, and this box~$X'$ is then also contained in~$C^{i-1}_u$. By \cref{constr:HPartition}, and because $C^{i-1}_u$ witnesses that $X, X'$ are contained in the same component of $G-\bigcup_{n \leq i-1} A_n$, the boxes $X, X'$ are at distance at least~$2d+2$ from each other in~$G$.
    \smallskip

    We now define a $u$--$o$ path in $G$ as follows. Let $ww'$ be an edge of $G$ between $C^i_u$ and~$X'$.
    Let $P_1$ be a path from $u$ to $w$ inside $C^i_u$. Further, let $P_2$ be a shortest $w'$--$o$ path in~$G$. Then $P_1 + ww' + P_2$ is a $u$--$o$ path in $G$. We claim that it avoids $S$. 

    Clearly, $P_1$ avoids $S$ since it is contained in $C^i_{u}$. Therefore, it remains to show that $P_2$ avoids $S$. For this, let us first note that since $S$ has diameter at most~$d$ and meets $X$, it is contained in $B_G(X, d)$. As $X, X'$ are contained in the same component of $G-\bigcup_{n \leq i-1} A_n$, they are at distance at least~$2d+2$ from each other, and hence $B_G(X', d+1)$ avoids $S$. 
    Therefore, the first $2d+1$ vertices of~$P_2$ avoid $S$. As $P_2$ is a shortest path from~$w'$ to~$o$, its remaining vertices are contained in $\bigcup_{n \leq i-2} A_n$, and hence also avoid $S$.
    \medskip

    \noindent \textbf{Case 2b}: \emph{$C^{i-1}_u$ does not meet $S$.}

    The reader may look at \cref{fig:BlocksAreInseparable:Case2b} to follow the proof of Case 2b more easily.

    Let $\defim{X} \in \cx$ be such that $S \cap X \neq \emptyset$ and $C_u^{i-1}$ sends an edge to~$X$ (if such a box exists, otherwise we let $X := \emptyset$). In particular, $X \subseteq A_{i-1}$.

    If $v$ lies in $C_u^{i-1}$, then there exists a $u$--$v$ path in $C_u^{i-1}$, which avoids $S$ by the assumption of Case~2b. Thus, we may assume that $v$ lies outside of~$C_u^{i-1}$. Since $u,v$ lie inside boxes belonging to the same block of~$H$, i.e.\ $h(u), h(v)$ lie in the same block of~$H$, they cannot be separated in~$G$ by removing a single box of~$\cx$. In particular, $C_u^{i-1}$ sends edges to at least one box~\defi{$X'$} in $A_{i-1}$ other than~$X$.

    Then $X'$ avoids $S$. Indeed, if $X=\emptyset$, then $X'$ avoids $S$ by the definition of~$X$. 
    Otherwise, if $X$ is a box in $\cx$, then $C_u^{i-1}$ witnesses that $X, X'$ are contained in the same component of $G - \bigcup_{n \leq i-2} A_n$, and thus $X,X'$ are at distance at least~$2d+2$ from each other in~$G$. Since $S$ has diameter at most~$d$ and meets~$X$, it follows that $S$ avoids~$X'$.
    \smallskip
    
    We now define a $u$--$o$ path in~$G$ as follows. Let $ww'$ be an edge of~$G$ between $C_u^{i-1}$ and~$X'$. Let $P_1$ be a path from~$u$ to~$w$ inside~$C_u^{i-1}$. Further, let $P_2$ be a shortest $w'$--$o$~path in~$G$. Then $P_1 + ww' + P_2$ is a $u$--$o$ path in $G$. We claim that it avoids~$S$.
    
    Clearly, $P_1$ avoids $S$ since it is contained in $C_u^{i-1}$, and $C_u^{i-1}$ avoids $S$ by the assumption of Case~2b. Moreover, as $P_2$ is a shortest path, its first $d$ vertices are contained in~$X'$, and its remaining vertices are contained in $\bigcup_{n \leq i-2} A_n$. As both $X'$ and $\bigcup_{n \leq i-2} A_n$ avoid~$S$, also $P_2$ avoids~$S$.
\end{proof}

Given a block $B$ of $H$ and a node $h \in V(B)$, let 
\[
\defim{ E_{Bh} }:=\left\{v \in X_h \mid vu\in E(G) \text{ for some } u \in \bigcup_{b \in V(B)\setminus\{h\}} X_b \right\}
\]
be the set of vertices of $X_h$ sending an edge to a box of $B-h$ (think of $E_{Bh}$ as the intersection of $B$ with $X_h$). We say that  $E_{Bh}$ is \defi{small}, if $\diam_G(E_{Bh}) \leq 5d+2$, otherwise we call it \defi{wide}. We call a cutvertex $h$ of $H$ \defi{small}, if $E_{Bh}$ is small for every block $B$ containing $h$. 

Given two blocks $B,B'$ of $H$ that intersect at a cutvertex $h$ of $H$, we say that $h$ \defi{identifies} $B,B'$, if both $E_{Bh}, E_{B'h}$ are wide. A \defi{block-class} is an equivalence class of blocks \wrt\ the reflexive transitive closure of the relation of being identified. 

We now define the desired \td\ $(T, \cv)$ of~$G$. 
\medskip

\noindent\textbf{Definition of $T$:} Let \defi{$T'$} be the block-cutvertex tree of $H$.
Let \defi{$T$} be the tree obtained from $T'$ by contracting each edge $Bh\in E(T')$ between a block~$B$ and a cutvertex~$h$ \st\ $E_{Bh}$ is wide. Note that each block-class $C$ of $H$ corresponds to exactly one vertex~\defi{$t_C$} of $T$, and distinct block-classes $C \neq C'$ correspond to distinct vertices $t_C \neq t_{C'}$. It is straightforward to check that 
\labtequ{T vces}{the vertices of $T$ that are not of the form $v_C$, for some block-class~$C$ of~$H$, are precisely the small cutvertices of $H$.} 
This completes the definition of~$T$.
\medskip

\noindent\textbf{Definition of the bags $V_t$:}
For every block~$B$ of~$H$, and every $h\in V(B)$, we define the \defim{$B$-part} of~$h$ to be $E_{Bh}$ if $h$ is a cutvertex of~$H$ and $E_{Bh}$ is small, and define the \defim{$B$-part} of~$h$ to be $X_h$ otherwise. (In both cases, the $B$-part of~$h$ is a subset of~$X_h$.)

We have two types of bags, corresponding to the two types of vertices of $T$:
\begin{enumerate}[label=\rm{(\Alph*)}]
    \item \label{itm:Def:Prebag:Cutv} For every small cutvertex $h$ of $H$, we define its bag by $\defim{V_h}:=X_h$. 
    
    \item \label{itm:Def:Prebag:Block} For every block-class~$C$ of~$H$, we define the bag \defi{$V_{t_C}$} of $t_C$ to consist of the union of all $B$-parts over all blocks $B\in C$. 
\end{enumerate}
Note that if $h$ is not a small cutvertex, then there exists some block~$B$ of~$H$ such that $h \in B$ and such that either $h$ is not a cutvertex, or $E_{Bh}$ is wide. In any case, the $B$-part of $h$ equals $X_h$, and thus $X_h$ is contained in the bag of~$C$ where $C$ is the block-class containing~$B$.
This completes the definition of the bags~$V_t$.
\medskip

We now show that the bags~$V_t$ define a \td\ $(T,\cv)$ of~$G$ over~$T$. Indeed, \ref{TDi} is satisfied as the boxes~$X_h$ form a partition of~$V(G)$, and because each box~$X_h$ is contained in some bag: of type~\ref{itm:Def:Prebag:Cutv} if $h$ is a small cutvertex of~$H$, and of type~\ref{itm:Def:Prebag:Block} otherwise. 
To check \ref{TDii}, note first that if $e=uv \in E(G)$ lies in a box~$X_h$, then it is contained in a bag by the previous argument. If, on the other hand, $X(u) \neq X(v)$, then the bag corresponding to the block-class of the block~$B$ of~$H$ containing $e$ contains the $B$-parts of both $X(u), X(v)$, and therefore both $u,v$. 

To check \ref{TDiii}, let $s,t \in V(T)$. If $s\neq t$ are small cutvertices of~$H$, then their bags are of type~\ref{itm:Def:Prebag:Cutv}, and therefore disjoint as $\cx$ is a partition. So we may assume that $s=t_C$ for a block-class~$C$ of~$H$. Then $V_{t_C}$ is of type~\ref{itm:Def:Prebag:Block}, and hence $V_{t_C} \subseteq \bigcup_{b \in \bigcup C} X_b$. Therefore, $V_{t_C}$ and $V_t$ can only intersect in the box $X_h$ of a cutvertex~$h$ of~$H$ such that $h \in \bigcup C$ and either $t=h$ or $h \in \bigcup C'$ such that $t = t_{C'}$. In particular, $t_C$ and $t$ are adjacent in~$T$. This completes the verification of~\ref{TDiii}.
\medskip

Next, we show that $(T, \cv)$ satisfies \ref{itm:TD i}. More precisely, 
\labtequ{diam adh}{every adhesion set $V_e$ of $(T, \cv)$ satisfies $\diam_G(V_e) \leq 5d+2$.} 
By \eqref{T vces} either $e$ connects two vertices corresponding to block classes $C\neq C'$ separated by a cutvertex $h$ of $H$ that does not identify $C$ with $C'$, or $e$ connects a  cutvertex $h$ of $H$ with a block class $C\ni B$ with $E_{Bh}$ small. If the former is the case, then one of $C,C'$ contains a block $B\ni h$ with $E_{Bh}$ small, and the other contains a block $B'\ni h$ with $E_{B'h}$ wide. In this case we have $V_e=E_{Bh}$. If the latter is the case, then again $V_e=E_{Bh}$. Thus, in both cases \eqref{diam adh} is confirmed as $E_{Bh}$ is small, and hence $\diam_G(E_{Bh}) \leq 5d+2$.
\medskip

We now show that $(T, \cp)$ satisfies \ref{itm:CoarseBlockButTree1:ii}, i.e.\ every bag $V_t$ is $(d,2d+1)$-inseparable. If $V_t$ is of type~\ref{itm:Def:Prebag:Cutv}, then this follows from \cref{lem:blocks}, and if $V_t$ is of type~\ref{itm:Def:Prebag:Block}, then this follows from the following lemma: 

\begin{lemma} \label{lem:block-classes}
    Every bag $V_{t_C}$ of type~\ref{itm:Def:Prebag:Block} is $(d,2d+1)$-inseparable in $G$.
\end{lemma}
%%%%%%%%%%%%%%%
\begin{proof} 
    Let $u,v \in V_{v_C}$, and $S\subseteq V(G)$ with $\diam_G(S)\leq d$ \st\ $d(\{u,v\},S) > 2d+1$. We need to show that $G-S$ contains a \pth{u}{v}.

    Let $P$ be a \pth{h(u)}{h(v)}\ in~$H$. Let $h_1, \ldots, h_k$ be the cutvertices of $H$ contained in $P$ other than $h(u),h(v)$ (if no such $h_i$ exists then $P$, and therefore $h(u),h(v)$, is contained in a block of $H$, and we are done by \Cref{lem:blocks}).
    Let $B_i^-,B_i^+$ be the blocks of $H$ containing the two edges of $P$ incident with $h_i$, and note that both $E_{B^-_i h_i}, E_{B^+_i h_i}$ are wide since $B^-_i,B^+_i$ lie in the block-class~$C$. Pick vertices $v^-_i,v^+_i$ of \G\ in $E_{B^-_i h_i}, E_{B^+_i h_i}$ respectively, so that both $v^\pm_i$ have distance at least $2d+2$ from~$S$, which is possible since $\diam_G(E_{B^\pm_i h_i}) \geq 5d+3$ and $\diam_G(S) \leq d$. 

    The resulting sequence of vertices $y v_1^- v_1^+v^-_2 \ldots v_k^+ z$ is at distance more than~$2d+1$ from~$S$, and any two consecutive members lie in a common block, and can therefore be connected by a path in~$G$ that avoids $S$ by \Cref{lem:blocks}.
\end{proof}

This completes the proof of \cref{main:CoarseBlockCutTree1}.
\bigskip

We now prove \cref{thm:cbct}. For this, we essentially take the same \td\ $(T, \cv)$ as in \cref{main:CoarseBlockCutTree1}, but we expand its bags $V_t$ slightly. We then show that the arising \td\ $(T, \cv')$ is already as desired for \cref{thm:cbct}.
\medskip

\noindent\textbf{Definition of the bags $V'_t$:} For each node~$t$ of~$T$, we extend its bag $V_t$ by $5d+3$ to obtain the bag $V'_t$, i.e.\ $\defim{V'_t}:= V_t^{+5d+3} = B_G(V_t, 5d+3)$.
\smallskip

By \cref{lem:TD exp}, $(T,\cv')$ is still a \TD\ of $G$. It follows immediately from~\eqref{diam adh} that each adhesion set $V'_e$ of $(T,\cv')$ has diameter less than $3\cdot(5d+3) = 15d+9$ (in~$G[V'_e]$), which proves \ref{itm:TD i}. 
The second statement of \cref{lem:TD exp} yields \ref{itm:TD ii}.

It remains to prove \ref{itm:TD iii}, i.e.\ every outer-torso $O'_t$ of a bag $V'_t$ is an $(d,7d+4)$-inseparable graph.
For this, we first show that every bag~$V'_t$ is $(d, 7d+4)$-inseparable.
Indeed, by \Cref{lem:blocks,lem:block-classes}, every (old) bag~$V_t$ is $(d,2d+1)$-inseparable in \G. Let $u',v' \in V'_t$ and $S \subseteq V(G)$ with $\diam_G(S) \leq d$ such that $d_G(\{u',v'\},S) > 7d+4$. Then there are $u,v \in V_t$ such that $d_G(u,u'), d_G(v,v') \leq 5d+3$. Since $V_t$ is $(d,2d+1)$-inseparable in $G$, either at least one of $u,v$, say $u$, is at distance at most~$2d+1$ from~$S$, or there is a $u$--$v$~path $P$ in $G$ that avoids $S$. In the former case, $d_G(u',S) \leq (2d+1) + (5d+3) = 7d+4$, a contradiction.
In the latter case, we can extend $P$ to a $u'$--$v'$~path~$P'$ by adding shortest $u'$--$u$ and $v'$--$v$ paths $P_u, P_v$. As $P_u,P_v$ have length at most $5d+3$, and $d_G(\{u',v'\},S) > 7d+4$, the path~$P'$ avoids~$S$.

Now to see that $O'_t$ is $(d, 7d+4)$-inseparable, let $u,v \in V(O'_t)$ and $S \subseteq V(O'_t)$ with $\diam_{O'_t}(S) \leq d$, and assume that $d_{O'_t}(\{u,v\}, S) > 7d+4$. If $S$ avoids  the set $B_G(V_t, 3d+3) \subseteq V'_t \subseteq V(O'_t)$, then $O'_t$ contains a $u$--$v$ path that avoids $S$: Since each adhesion set of $(T, \cv)$ has diameter at most $5d+2$ in~$G$, it follows that $O'_t[B_G(V_t, 3d+3)] \supseteq G\left[B_G\left(V_t, \left\lceil\frac{5d+2}{2}\right\rceil\right)\right]$ is connected. Thus, $O'_t[B_G(V_t, 3d+3)]$ contains a $u$--$v$ path, which avoids $S$ as $S$ avoids $B_G(V_t, 3d+3)$. 

Hence, we may assume that $S$ meets $B_G(V_t, 3d+3)$. Then, as $V'_t = B_G(V_t,5d+3)$, we have $d_G(S, G-V'_t) \geq d$, and therefore $\diam_G(S) \leq d$. Since $O'_t$ is a contraction minor of~$G$, we also have $d_G(\{u,v\},S) > 7d+4$. Thus, as $V'_t$ is $(d,7d+4)$-inseparable in~$G$ as shown above, it follows that there is a $u$--$v$ path $P$ in $G$ that avoids $S$. Since $S \subseteq B_G(V_t, 4d+3) \subseteq V'_t$, it follows that the path $P'$ in $O'_t$ induced by~$P$ still avoids $S$ (as $S$ avoids all contraction vertices of~$O'_t$).

This completes the proof of \cref{thm:cbct}.

\section*{Acknowledgments}

We thank Chiara Molinari for helpful discussions.

\bibliographystyle{plain}
\bibliography{collective}

\begin{thebibliography}{10}

\bibitem{ADEFJKW}
S.~Albrechtsen, R.~Diestel, A.-K. Elm, E.~Fluck, R.~W. Jacobs, P.~Knappe, and P.~Wollan.
\newblock A structural duality for path-decompositions into parts of small radius.
\newblock {\em Innovations in Graph Theory}, 3:207--246, 2026.

\bibitem{ADGK2t}
S.~Albrechtsen, M.~Distel, and A.~Georgakopoulos.
\newblock {Excluding $K_{2,t}$ as a fat minor}.
\newblock arXiv:2510.14644.

\bibitem{ADGSmallCounterexamples}
S.~Albrechtsen, M.~Distel, and A.~Georgakopoulos.
\newblock Small counterexamples to the fat minor conjecture.
\newblock arXiv:2601.05761.

\bibitem{AHJKW}
S.~Albrechtsen, T.~Huynh, R.~W. Jacobs, P.~Knappe, and P.~Wollan.
\newblock {A Menger-Type Theorem for Two Induced Paths}.
\newblock {\em SIAM J.\ Discrete Math.}, 38(2):1438--1450, 2024.

\bibitem{AJKW}
S.~Albrechtsen, R.~Jacobs, P.~Knappe, and P.~Wollan.
\newblock A characterisation of graphs quasi-isometric to ${K}_4$-minor-free graphs.
\newblock {\em Combinatorica}, 45:61, 2025.

\bibitem{BBCPPCoarseBlockCutTree}
J.~Balig\'acs, V.~Bla{\v{z}}ej, J.~Czy{\.{z}}ewska, Mi. Pilipczuk, and E.~Protopapas.
\newblock A coarse block-cut tree theorem.
\newblock preprint, 2026.

\bibitem{BLPPSep}
E.~Bonnet, H.~Le, Ma. Pilipczuk, and Mi. Pilipczuk.
\newblock {Coarse Balanced Separators in Fat-Minor-Free Graphs}.
\newblock arXiv:2604.11318.

\bibitem{CDNRV}
V.~Chepoi, F.~F. Dragan, I.~Newman, Y.~Rabinovich, and Y.~Vax\`es.
\newblock Constant {Approximation} {Algorithms} for {Embedding} {Graph} {Metrics} into {Trees} and {Outerplanar} {Graphs}.
\newblock {\em Discrete \& Computational Geometry}, 47(1):187--214, 2012.

\bibitem{DGHLMGallai}
M.~Distel, U.~Giocanti, J.~Hodor, C.~Legrand-Duchesne, and P.~Micek.
\newblock {A coarse Gallai theorem}.
\newblock arXiv:2601.18439.

\bibitem{FujPapCoa}
K.~Fujiwara and P.~Papasoglu.
\newblock A coarse-geometry characterization of cacti.
\newblock {arXiv:2305.08512}.

\bibitem{FujPapAsy}
K.~Fujiwara and P.~Papasoglu.
\newblock Asymptotic dimension of planes and planar graphs.
\newblock {\em Trans.\ Am.\ Math.\ Soc.}, 374:8887--8901, 2021.

\bibitem{GeoMolQua}
A.~Georgakopoulos and C.~Molinari.
\newblock {Quasi-isometries, contractions, and intersection graphs}.
\newblock {In preparation}.

\bibitem{GeoPapMin}
A.~Georgakopoulos and P.~Papasoglu.
\newblock {Graph minors and metric spaces}.
\newblock {\em Combinatorica}, 45:33, 2025.

\bibitem{LiuWeakCoarseMenger}
C.-H. Liu.
\newblock Coarse menger property of quasi-minor excluded graphs and length spaces.
\newblock arXiv:2605.10068.

\bibitem{McMFat}
J.~MacManus.
\newblock Fat minors in finitely presented groups.
\newblock {\em Combinatorica}, 45:40, 2025.

\bibitem{NgScSeCoa}
T.~Nguyen, A.~Scott, and P.~Seymour.
\newblock Asymptotic structure. {I}. {C}oarse tree-width.
\newblock arXiv:2501.09839.

\bibitem{NgScSeAsyIV}
T.~Nguyen, A.~Scott, and P.~Seymour.
\newblock {Asymptotic structure.\ IV.\ A counterexample to the weak coarse Menger conjecture}.
\newblock arXiv:2508.14332.

\bibitem{Tutte}
W.T. Tutte.
\newblock {\em {Connectivity in graphs.}}
\newblock {Mathematical Expositions. 15. Toronto: University of Toronto Press; London: Oxford University Press. IX, 145 p.}, 1966.

\end{thebibliography}

\end{document}